\documentclass{scrartcl}
\usepackage[english]{babel}
\usepackage[utf8]{inputenc}
\usepackage[T1, T2A]{fontenc}

\usepackage{enumerate}
\usepackage{amsmath}
\usepackage{amssymb}
\usepackage{amsthm}
\usepackage{mathtools}
\usepackage{todonotes}
\usepackage{hyperref}

\theoremstyle{definition}
\newtheorem{definition}{Definition}[section]
\theoremstyle{proposition}
\newtheorem{proposition}[definition]{Proposition}
\newtheorem{theorem}[definition]{Theorem}
\newtheorem{lemma}[definition]{Lemma}
\newtheorem{corollary}[definition]{Corollary}
\theoremstyle{remark}
\newtheorem{remark}[definition]{Remark}
\newtheorem{example}[definition]{Example}

\numberwithin{equation}{section}

\DeclareMathOperator{\dist}{dist}

\title{On the rate of convergence of iterated Bregman projections and of the alternating algorithm}
\author{Christian Bargetz \and Emir Medjic}

\begin{document}
\maketitle
\begin{abstract}
  \noindent\textbf{\textsf{Abstract.}} We study the alternating algorithm for the computation of the metric projection onto the closed sum of two closed subspaces in uniformly convex and uniformly smooth Banach spaces. For Banach spaces which are convex and smooth of power type, we exhibit a condition which implies linear convergence of this method. We show these convergence results for iterates of Bregman projections onto closed linear subspaces. Using an intimate connection between the metric projection onto a closed linear subspace and the Bregman projection onto its annihilator, we deduce the convergence rate results for the alternating algorithm from the corresponding results for the iterated Bregman projection method.
\end{abstract}

\section{Introduction}

Let $X$ be a Banach space and let $M$ be a closed subspace. We denote by
\[
  P_{M}x := \{y\in M\colon \|x-y\|=d(x,M)\}
\]
the set of points which realise the distance between the subspace $M$ and the point $x\in X$. In general Banach spaces this set may be empty. On the other hand, it is well known that for reflexive spaces the set $P_{M}$ is always nonempty. If $X$ is, in addition, strictly convex, i.e., the unit sphere $S_X$ of $X$ does not contain any nontrivial segments, the set $P_{M}(x)$ consists of one point only. Therefore, for reflexive strictly convex Banach spaces and for every closed subspace $M$, we can consider the mapping
\begin{equation}\label{eq:MetricProjMap}
  P_M\colon X\to X, x\mapsto P_{M}x.
\end{equation}
In Hilbert spaces the mapping $P_M$ coincides with the orthogonal projection onto $M$.

In the context of the development of the theory of linear operators on Hilbert spaces, J.~von~Neumann showed in~\cite{vNe1949:RingsOfOperators} that given two closed subspaces $M,N\subset H$ of a Hilbert space $H$ with orthogonal projections $P_M$ and $P_N$, the sequence
\begin{equation}\label{eq:VonNeumannAlg}
  x_0:=x, \qquad x_{2n+1} := P_Mx_{2n} \quad\text{and}\quad x_{2n}:=P_Nx_{2n-1}
\end{equation}
converges to $P_{M\cap N}(x)$ for all $x\in H$. Algorithm~\eqref{eq:VonNeumannAlg} is called the \emph{von~Neumann alternating projection algorithm}. An elementary and geometric proof of this result was given by E.~Kopeck\'{a} and S.~Reich in~\cite{KR2004:VonNeumann}. Von~Neumann's convergence  result was generalised to the case of a finite number of orthogonal projections by I.~Halperin in~\cite{Hal1962:Products}. It should be noted that here it is important that the iterations are cyclic, i.e., for subspaces $M_1,\ldots,M_k$ the operator $P_{M_k}\cdots P_{M_1}$ is iterated. Although it is possible to weaken this condition, it cannot be dropped altogether as has been shown in~\cite{KM2014:ThreeProjections} where, strengthening a result from~\cite{Pas2012:AmemiyaAndoFails}, the authors show that there is an iterative sequence of three projections which does not converge in the strong operator topology. These results settled in the negative a long standing conjecture of I.~Amemiya and T.~And\^o which was stated in~\cite{AA1965:Convergence}.

In addition to the iterative method for finding a projection onto the intersection of subspaces, J.~von~Neumann also showed that for closed subspaces $M, N\subset H$ the sequence defined by
\begin{equation}
  x_0:=x, \qquad x_{2n+1} := x_{2n}-P_Mx_{2n} \quad\text{and}\quad x_{2n}:=x_{2n-1}-P_Nx_{2n-1}
\end{equation}
converges to $P_{(M+N)^\perp}x$ for all $x\in H$; see~\cite[Corollary, p.~56]{Neu1950:FunctionalOperators}. In other words, the sequence $\{((I-P_M)(I-P_N))^{n}\}_{n=0}^{\infty}$ converges to the projection $P_{(M+N)^\perp}$ in the strong operator topology.

Outside Hilbert spaces the situation is much more complicated. First of all the metric projection need not be linear. In fact, R.~A.~Hirschfeld showed in~\cite{Hir1958:BestApproximationII} that if $X$ is a Banach space which is at least three-dimensional and for every one-dimensional subspace $M$ the metric projection $P_{M}$ is additive, i.e., if it satisfies $P_{M}(x+y)=P_M(x)+P_M(y)$, then $X$ is an inner product space. In other words this means that having linear projections is a property that characterises inner-product spaces.

Moreover, for general reflexive and strictly convex spaces the metric projection need not be continuous; see~\cite{Bro1973:RotundReflexive}. On the other hand if $X$ is a uniformly convex Banach space, the metric projection is always continuous; see e.g.~\cite[Prop.~3.2]{GoebelReich1984}. Nevertheless, in contrast to the situation in Hilbert spaces, where the metric projection even on closed convex subsets $C$ is nonexpansive, this is not the case even for subspaces of uniformly convex spaces. In fact, J.~Lindenstrauss showed in~\cite{Lin1964:NonlinearProjections} that in many cases the existence of a uniformly continuous projection already implies the existence of a linear projection. Hence in these cases the metric projection on non-complemented subspaces cannot be even uniformly continuous.

The analysis of the alternating approximation method in Banach spaces started when in~\cite{Hir1958:BestApproximationII}, R.~A.~Hirschfeld posed the problem whether the property that for every pair of subspaces $M,N\subset X$ and all $x\in X$ the sequence
\begin{equation}\label{eq:AltMethod}
  ((I-P_M)(I-P_N))^nx \longrightarrow (I-P_{M+N})x
\end{equation}
converges, characterises the Hilbert spaces among Banach spaces. The first negative answer, in the two-dimensional case, was given by V.~Klee in~\cite{Kle1963:Hirschfeld} but Klee also remarked that in higher finite dimensions the situation might be markedly different. In~\cite{Sti1965:Hirschfeld}, W.~J.~Stiles answered this question in the negative by showing that in finite dimensional strictly convex spaces $X$ the sequence in~\eqref{eq:AltMethod} converges for all pairs of subspaces. Although his proof did not extend to the infinite dimensional setting, Stiles conjectured in the aforementioned paper that convergence in~\eqref{eq:AltMethod} holds in every uniformly convex and uniformly smooth Banach space. Note that~\eqref{eq:AltMethod} allows for an iterative algorithm to compute the best approximation of $x\in X$ in the space $M+N$, using only the metric projections onto $M$ and $N$. A generalisation of these results can be found in~\cite{Fra1973:AlternatingMethod}, where also weak convergence of the sequence in~\eqref{eq:AltMethod} is considered. In~1979, F.~Deutsch showed in~\cite{Deutsch1979} that under the assumption that the sum $M+N$ is a \emph{closed} subspace of a uniformly convex and uniformly smooth Banach space, the sequence in~\eqref{eq:AltMethod} converges for every $x\in X$. In the aforementioned paper, Deutsch raised the question of whether the condition that the sum of the spaces has to be closed is necessary. Very recently, this result was extended by A.~Pinkus in~\cite{Pinkus2015} to finitely many subspaces of uniformly convex and uniformly smooth Banach spaces. The question of the necessity of the closedness of the sum remains open. Summing up, the currently known properties of the alternating approximation method are the following.

\begin{theorem}[{Theorem~2.2 in~\cite{Pinkus2015}}]
  Let $X$ be a uniformly convex and uniformly smooth Banach space, and let $M_1,\ldots, M_k$ be closed linear subspaces so that $M_1+\cdots+M_k$ is closed. Then the sequence
  \[
    \{(I-P_{M_k})(I-P_{M_{k-1}})\cdots(I-P_{M_{1}})^{n}x\}_{n=0}^\infty
  \]
  converges to $x-P_{M_1+\cdots+M_k}x$ for all $x\in X$.
\end{theorem}

In~\cite{Pinkus2015}, A.~Pinkus poses the question of what can be said about the rate of convergence of this sequence. One of the aims of this article is to give a partial answer to this question for the case of two subspaces. Note that in the case where the metric projections are linear, the assumption that $M+N$ has to be closed can be dropped, see e.g.~\cite{Deutsch1979, Reich1982}. A similar result is true for any finite number of subspaces, see e.g.~\cite{Rei1983:LimitTheorem}.

The Bregman distance and Bregman projections were introduced by L.~Bregman in~\cite{Bregman1967} in the context of the common fixed point problem. These concepts turned out to be very useful for different problems in nonlinear analysis, including the convex feasibility problem and for the investigation of maximally monotone operators, see e.g.~\cite{Alber1997, Suantai2012, MartinMarquez2013}. The properties of these projections have been studied in detail by many authors, see e.g.~\cite{Resmerita2004, Schoepfer2008}. Algroithms including Bregman projections and iterated Bregman projections can be found in for example in~\cite{Alber1996, BauschkeEtAl2015}. A weak convergence theorem for Bregman strongly nonexpansive mappings is given in~\cite{Reich1996}. Very recent results on Bregman distances in a more general setting, can be found in~\cite{Reem2019}.

We will exhibit an intimate connection between the alternating approximation method and iterations of Bregman projections. We give regularity conditions for closed linear subspaces which ensure linear convergence of iterated Bregman projections. These conditions can be translated to provide linear convergence of the alternating algorithm.

\section{Preliminaries and Notation}

Let $X$ be a Banach space with norm $\|\cdot\|$. Recall that the modulus of convexity of $X$ is defined as
\[
  \delta_{X}\colon [0,2]\to[0,\infty), \qquad \varepsilon\mapsto \inf \left\lbrace \left. 1 - \frac{\|x+y\|}{2} \ \right| \ x,y\in S_X, \|x-y\| \geq \varepsilon \right\rbrace.
\]
The space $X$ is said to be uniformly convex if $\delta_{X}(\varepsilon)>0$ whenever $\varepsilon>0$. Note that a uniformly convex space is in particular strictly convex. 

The modulus of smoothness of $X$ is defined as
\[
  \rho_{X}\colon [0,1]\to[0,\infty), \qquad \tau\mapsto  \sup \left\lbrace \left.\frac{ \|x+\tau y\| + \|x-\tau y\|}{2}-1 \ \right| \ x,y\in S_X  \right\rbrace
\]
and $X$ is said to be uniformly smooth if
\[
  \lim_{\tau\to 0} \frac{\rho_{X}(\tau)}{\tau} = 0.
\]
Note that the norm of a uniformly smooth space is in particular Fr\'{e}chet differentiable away from the origin. The concept of uniform convexity and uniform smoothness are dual to each other in the sense that the dual space of a uniformly convex space is uniformly smooth and the other way round. Note that these spaces are in particular reflexive. We also consider the following quantitive versions of the above definitions: A Banach space $X$ is called $\rho$-convex if there is a constant $c>0$ such that
\[
  \delta_{X}(\varepsilon) \geq c \varepsilon^\rho
\]
and $\sigma$-smooth if there is a constant $C>0$ such that
\[
  \rho_{X} (\tau) \leq C \tau^\sigma
\]
where $\sigma>1$.
Note that every $\rho$-convex space is uniformly convex and every $\sigma$-smooth space is uniformly smooth. For more detailed information on uniformly convex and uniformly smooth spaces, we refer the interested reader to~\cite{Cioranescu1990,GoebelReich1984,Rei1992:Review}.

Let $X$ be a uniformly convex and uniformly smooth Banach space. We denote by
\[
  j_p\colon X\to X^*
\]
the duality mapping defined by
\[
  j_p(x) = D\left(\frac{1}{p}\|\cdot\|^{p}\right)(x)
\]
which is well-defined since the norm of a uniformly smooth Banach space is in particular Fr\'{e}chet differentiable. We will use thoughout the article that the mapping $j_p$ is bijective and its inverse is the duality mapping $j_{p^*}\colon X^*\to X$ where $\frac{1}{p}+\frac{1}{p^*}=1$. For a detailed examination of the properties of duality mappings, we refer the interested reader to Chapters~I and~II of~\cite{Cioranescu1990}, cf.~\cite{Rei1992:Review}. A brief overview of important properties of the duality mapping can be found in Section~2.2 of~\cite{Schuster2012}.

\begin{definition} \label{def:bregman-distance}
  Let $X$ be a uniformly convex and uniformly smooth Banach space and $x,y\in X$. We define by
  \[
    D_p(x,y) = \frac 1 p \|x\|^p - \frac 1 p \|y\|^p - \langle j_p(y), x-y \rangle
  \]
  the \emph{Bregman distance} of $x$ and $y$.
\end{definition}

\begin{remark}
  The Bregman distance is well defined, since by assumption $\|\cdot\|$ is Fr\'{e}chet differentiable, hence $\frac{1}{p} \|\cdot\|^p$ is differentiable too. Moreover the convexity of this mappings implies that $D_p(x,y)\geq 0$ for all $x,y\in X$.\\
  Note that some authors prefer to use the term \emph{Bregman divergence} instead of Bregman distance since it neither is symmetric nor does it satisfy the triangle inequality. It does however satisfy the so called three point identity
  \begin{equation}\label{eq:ThreePointIdentity}
    D_p(x,y) = D_p(x,z) + D_p(z,y) + \langle j_p(z)-j_p(y),x-z\rangle,
  \end{equation}
  for $x,y,z\in X$, which can be thought of a substitute for the triangle inequality. Moreover the Bregman distance from a point $x$ to a point $y$ is zero if and only if these points coincide.
\end{remark}

For the convenience of the reader, let us recall a few important properties of the Bregman distance.

\begin{proposition}\label{prop:BregmanDistanceProperties}
  Let $X$ be a uniformly convex and uniformly smooth Banach space.
  \begin{enumerate}[(i)]
  \item The mapping
    \[
      X\times X\to[0,\infty), \qquad (x,y) \mapsto D_p(x,y)
    \]
    is continuous.
  \item A sequence $(x_n)_{n=0}^{\infty}$ converges to $x\in X$ if and only if $D_p(x, x_n)\to 0$.
  \item A sequence $(x_n)_{n=0}^{\infty}$ is a Cauchy sequence in $X$ if and only if for every $\varepsilon>0$ there is an $N\in\mathbb{N}$ such that
    \[
      D_p(x_m,x_n) < \varepsilon
    \]
    for all $m,n\geq N$.
  \item Let $(x_n)_{n=0}^\infty$ and $(y_n)_{n=0}^{\infty}$ be two sequences in $X$ and assume that $(x_n)_{n=0}^{\infty}$ is bounded. Then,
    \[
      D_p(y_n,x_n) \to 0 \Rightarrow \|x_n-y_n\|\to 0.
    \]
  \item Denoting by $D_{p^*}(\cdot,\cdot)$ the Bregman distance on the dual space $X^*$ with respect to the exponent $p^*$ satisfying $\frac{1}{p}+\frac{1}{p^*}=1$, the Bregman distance satisfies
    \[
      D_p(x,y) = D_{p^{*}}(j_p(y),j_p(x))
    \]
    for all $x,y\in X$.
  \end{enumerate}
\end{proposition}

\begin{proof}
  See e.g. Theorem~2.60 and Lemma~2.63 in~\cite[pp.~45--46]{Schuster2012}. Assertion~(iv) is part of Corollary~2.4 in~\cite{Butnariu2000}.
\end{proof}

\begin{definition}[Bregman projection]\index{Bregman-!projection}
  Let $M$ be a closed linear subspace of $X$. We denote by
  \[
    \Pi_M^p x = \underset{m\in M}{\mathrm{arg \ min}}\ D_p(m,x)
  \]
  the \emph{Bregman projection} from $x$ onto $M$. We denote by
  \[
    D_p(M,x) := D_p(\Pi_M^p x,x)
  \]
  the Bregman distance of $x$ to the subspace $M$.
\end{definition}

Note that by Proposition 3.6 in~\cite{Schoepfer2008} the Bregman projection satisfies
\begin{equation}\label{eq:BregmanNorm}
  \|\Pi^p_{M}x\| \leq \|x\|
\end{equation}
for all $x\in X$. In addition, we will need a few properties well-known of the Bregman projection which we repear here for the convenience of the reader.

\begin{proposition}[{Proposition 3.7 in~\cite[p.10]{Schoepfer2008}}]\label{prop:bregmanprojAnnihilator}
  Let $M\subset X$ be a closed linear subspace and $x,y\in X$ be given. Then the following assertions are equivalent:
  \begin{enumerate}[(i)]
  \item $y = \Pi_M^p x$
  \item $y \in M$ and $j_p(y) - j_p(x) \in M^\perp$
  \end{enumerate}
\end{proposition}

This characterisation in terms of the annihilator of $M$ can be thought of as a generalisation of the characterisation of the orthogonal projection by orthogonality of image and kernel.

\begin{proposition}[{Corollary 2.2 in~\cite[p.~41]{Alber1997}}]\label{prop:Bregman-SQNE}
  Let $C\subset X$ be a closed, convex subset. The Bregman projection onto $C$ is Breman strongly quasi-nonexpansive, i.e. it satisfies
  \begin{align}\label{eq:sqne-ineq}
    D_p(z,\Pi_C^p x) \leq D_p(z,x) - D_p(\Pi_C^p x, x).
  \end{align}
  for all $z\in C$ and all $x\in X$. In the case where $C$ is not only convex but a linear subspace the above is true with equality instead of inequality.
\end{proposition}

In~\cite{XuRoach1991}, Z.~B.~Xu and G.~F.~Roach gave characterisations of uniformly convex and uniformly smooth spaces in terms of inequalities including powers of the norm. Since we make frequent use of these inequalities, we repeat them here for the convenience of the reader.

\begin{theorem}[{Theorem~1 and Theorem~2 in~\cite[p.~195 and p.~204]{XuRoach1991}}]\label{thm:xu-roach}
  Let $X$ be a Banach space and let $1<p<\infty$. Then the following assertions are equivalent
  \begin{enumerate}[(i)]
  \item $X$ is uniformly smooth,
  \item the following inequality holds for all $x,y\in X$ with $\|x\|+\|y\| \not = 0$:
    \begin{equation} \label{eq:xu-roach:smooth}
      \|x + y\|^p \leq \|x\|^p + p \langle j_p(x), y \rangle + \sigma_p(x, y),
    \end{equation}
    where \[\sigma_p(x,y) \leq p \int_0^1 \frac{(\max\{\|x+ty\|, \|x\|\})^p}{t} K \rho_X\left( \frac{t \|y\|}{\max\{\|x+ty\|, \|x\|\}} \right)\mathrm{d} t, \] for some $K>0$.
  \end{enumerate}
  The following assertions are also equivalent 
  \begin{enumerate}[(i)]
  \item $X$ is uniformly convex,
  \item the following inequality holds for all $x,y\in X$ with $\|x\|+\|y\| \not = 0$:
    \begin{equation} \label{eq:xu-roach:convex} 
      \|x + y\|^p \geq \|x\|^p + p \langle j_p(x), y \rangle + \sigma_p(x, y),
    \end{equation}
    where \[\sigma_p(x,y) \geq p \int_0^1 \frac{(\max\{\|x+ty\|, \|x\|\})^p}{t} K \delta_X\left( \frac{t \|y\|}{2\max\{\|x+ty\|, \|x\|\}} \right)\mathrm{d} t, \] for some $K>0$,
  \item the following inequality holds for all $x,y\in X$ with $\|x\|+\|y\| \not = 0$:
    \begin{align} \label{eq:xu-roach-totalconv}
      \langle j_p(x)-j_p(y), x-y\rangle \geq (\max\{\|x\|, \|y\|\})^pK \delta_X \left( \frac{\|x-y\|}{2\max\{\|x\|,\|y\|\}}\right),
    \end{align}  
    for some $K>0$.
  \end{enumerate}
\end{theorem}

\begin{remark}
  The equivalences given above are just two particular cases of the characterisations given in~\cite{XuRoach1991}.
\end{remark}

The inequalities above can be use the obtain inequalities between the Bregman distance and powers of the norm distance.

For the particular case of $\rho = p$ and $\sigma = p$ the following proposition is contained in Theorem~2.60 in~\cite[pp.~45--46]{Schuster2012}.

\begin{proposition}
  \label{coro:bregman:sym}
  Let $X$ be a Banach space and $R>0$.
  If $X$ is $\rho$-convex, i.e. uniformly convex of power type $\rho$, and $p\leq \rho$, then the inequality
  \begin{align}\label{eq:bregman:sym:1}
    D_p(x,y) \geq C_\rho \|x -y\|^\rho
  \end{align} 
  holds for all $x,y \in X$ with $\|x\|, \|y\| < R$.\\
  If $X$ is $\sigma$-smooth, i.e. uniformly smooth of power type $\sigma$, and $p\geq \sigma$, then the inequality
  \begin{align}\label{eq:bregman:sym:2}
    D_p(x,y) \leq C_\sigma \| x -y \|^\sigma
  \end{align}
  holds for all $x,y \in X$ with $\|x\|, \|y\| < R$.\\
  If $X$ is both, $\rho$-convex and $\sigma$-smooth, and $\sigma \leq p \leq \rho$, then the inequality
  \begin{align}\label{eq:bregman:sym:3}
    D_p(x,y) \leq \frac{C_\sigma}{C_\rho^\frac{1}{\rho}} D_p(y,x)^{\frac{\sigma}{\rho}}
  \end{align}
  holds for all $x,y \in X$ with $\|x\|, \|y\| < R$.
\end{proposition}

\begin{proof}
  For $x,y\in X$ set $\bar x = x$ and $\bar y = y-x$ and with $X$ being $\sigma$-smooth, we conclude from inequality~\eqref{eq:xu-roach:smooth} that
  \begin{align*}  
    \|y\|^p &\leq \|x\|^p + p\langle j_p(x), y-x\rangle + \\
            & \qquad  + p\int_0^1 \frac{K}{t} (\max\{\|(1-t)x + ty\|, \|x\|\})^p \rho_X\left( \frac{t \|y-x\|}{\max\{\|(1-t)x+ty\|, \|x\|\}} \right) \mathrm{d}t \\
            &\leq \|x\|^p + p\langle j_p(x), y-x\rangle + p\int_0^1 K C R^{p-\sigma} t^{\sigma-1} \|y-x\|^\sigma \mathrm{d}t\\
            & = \|x\|^p + p\langle j_p(x), y-x\rangle + \frac p \sigma K C R^{p-\sigma}\|y-x\|^\sigma.
  \end{align*}
  Rearranging the terms above, we obtain
  \[
    p D_p(y,x) = \|y\|^p-\|x\|^p - p\langle j_p(x), y-x\rangle \leq  \frac p \sigma K C R^{p-\sigma}\|y-x\|^\sigma,
  \]
  where $C > 0$ is chosen so that $\rho_X(t) \leq C t^\sigma$.
  
  In the case that $X$ is $\rho$-convex we obtain from inequality~\eqref{eq:xu-roach:convex} that,
  \begin{align*}  
    \|y\|^p &\geq \|x\|^p + p\langle j_p(x), y-x\rangle + \\
            & \qquad + p\int_0^1 \frac{K}{t} (\max\{\|(1-t)x + ty\|, \|x\|\})^p \delta_X\left( \frac{t \|y-x\|}{2\max\{\|(1-t)x+ty\|, \|x\|\}} \right) \mathrm{d}t \\
            &\geq \|x\|^p + p\langle j_p(x), y-x\rangle + p\int_0^1 K C' R^{p-\rho} t^{\rho-1} 2^{-\rho} \|y-x\|^\rho \mathrm{d}t\\
            & = \|x\|^p + p\langle j_p(x), y-x\rangle + \frac p \rho K C' 2^{-\rho} R^{p-\rho}\|y-x\|^\rho.
  \end{align*}
  Again, rearranging the terms yields
  \[
    p D_p(y,x) = \|y\|^p-\|x\|^p - p\langle j_p(x), y-x\rangle \geq  \frac p \rho K C' 2^{-\rho} R^{p-\rho}\|y-x\|^\rho,
  \]
  where the constant $C'>0$ is chosen so that $\delta_X(t) \geq C' t^\rho$ and we used the fact that $p-\rho <0$ in the second inequality.
  
  The third inequality is a direct consequence of the first two.
\end{proof}

The above proposition holds for $x,y$ in a bounded subset of $X$. In a special case, we are able to get rid of this boundedness assumption.
\begin{corollary}\label{coro:bregman:sym:2-sc}
  Let $X$ be $2$-convex and $2$-smooth Banach space. Then the inqualities
  \[
    D_2(x,y) \leq C \|x-y\|^2 \quad \text{and} \quad D_2(x,y) \geq C' \|x-y\|^2
  \]
  hold globally. Moreover, there is a constant $\tilde{C}>0$ such that
  \[
    D_2(x,y) \leq \tilde{C} D_2(y,x).
  \]
\end{corollary}

\begin{proof}
  The constants $C_\sigma, C_\rho$ in Proposition~\ref{coro:bregman:sym} depend on $R$ by a factor of the form $R^{p-\sigma}$ or $ R^{p-\rho}$. Since we assume that $p=\rho=\sigma=2$ this dependence vanishes and these inequalities hold globally.
\end{proof}

\begin{remark}
  Note that every $2$-convex and $2$-smooth Banach space is isormorphic to a Hilbert space, see e.g. Theorem~2.9 in~\cite[p.~43]{BenyaminiLindenstrauss}.
\end{remark}

We can now use these inequalities to establish a connection between the norm distance to a subspace and the Bregman distance to this subspace.

\begin{lemma}\label{lem:bregmanproj-metricproj:ineq}
  Let $X$ be uniformly convex of power type $\rho$, uniformly smooth of power type $\sigma$ and $\sigma \leq p \leq \rho$. Then for every $R>0$ there are constants $C_1,C_2>0$, such that the inequalities
  \begin{align}\label{eq:bregmanproj-metricproj:1}
    C_1 D_p(\Pi_M^p x ,x )^\frac{1}{\rho} \geq \mathrm{dist}(x,M)
  \end{align}
  and
  \begin{align}\label{eq:bregmanproj-metricproj:2}
    \mathrm{dist}(x,M) \geq C_2 D_p(\Pi_M^p x, x)^\frac{1}{\sigma}
  \end{align}
  hold for all $x \in X$ with $\|x\|<R$.
\end{lemma}

\begin{proof}
  Let~$x\in X$ with $\|x\|<R$. Since $\|\Pi_M^p x \| \leq \|x\|$, we also have $\|\Pi_M^{p}x\|<R$. From inequality~\eqref{eq:bregman:sym:1} in Corollary~\ref{coro:bregman:sym} we may conclude that 
  \[
    \mathrm{dist}(x,M) = \|P_M x - x\| \leq \|\Pi_M^p x - x \| \leq C_1 D_p(\Pi_M^p x ,x )^\frac{1}{\rho}.
  \]  
  The second equations follows analogously from inequality~\eqref{eq:bregman:sym:2} in Corollary~\ref{coro:bregman:sym}.
\end{proof}

\section{Bregman regularity properties of pairs of subspaces}

\begin{definition} \label{def:bregman-regularities}
  We say the pair $(M,N)$ is \emph{Bregman regular}, if for every $\varepsilon >0$, there is a $\delta>0$, such that for all $x\in X$ the implication
  \begin{equation}
    \max \left\lbrace D_p(M,x), D_p(N,x)\right\rbrace < \delta \Rightarrow D_p(M\cap N, x) < \varepsilon,
  \end{equation}
  holds. The pair $(M,N)$ is said to be \emph{boundedly Bregman regular}, if for every bounded set $S\subset X$ and every $\varepsilon>0$, there is a $\delta >0$ such that 
  \begin{equation}
    \max \left\lbrace D_p(M,x), D_p(N,x)\right\rbrace < \delta \Rightarrow D_p(M\cap N, x) < \varepsilon,
  \end{equation}
  holds true that. $(M,N)$ is called \emph{linear Bregman regular}, if there is constant $\kappa >0$ such that 
  \begin{equation}
    D_p(M\cap N,x) \leq \kappa \max\{D_p(M,x), D_p(N,x)\},
  \end{equation}
  for all $x\in X$.
  Finally, $(M,N)$ is called \emph{boundedly linear Bregman regular}, if for every bounded set $S\subset X$ there is constant $\kappa >0$ such that 
  \begin{equation}
    D_p(M\cap N,x) \leq \kappa \max\{D_p(M,x), D_p(N,x)\}
  \end{equation}
  for all $x\in S$.
\end{definition}
\begin{proposition}\label{prop:bregman:regularity}
  For a pair $(M,N)$ of closed linear subspaces the following assertions are equivalent:
  \begin{enumerate}[(i)]
  \item $(M,N)$ is Bregman regular,
  \item $(M,N)$ is linear Bregman regular,
  \item $(M,N)$ is boundedly linear Bregman regular.
  \end{enumerate}
\end{proposition}

The proof of this equivalences follows the ideas of the proof of equivalence of the corresponding ordinary regularity properties given in~\cite[p.~196]{Bauschke1993}.

\begin{proof}
  We first show that assertion (i) implies (ii). With this aim, let $\varepsilon >0$. By assumtion there is a $\delta > 0$, such that
  \[
    m(x) < \delta \Rightarrow D_p(\Pi_{M\cap N}^px,x) < \varepsilon
  \]
  where
  \[
    m(x) := \max \{D_p(\Pi_M^p x, x), D_p(\Pi_N^px, x)\}.
  \]
  Now let $y\in X \setminus (M\cap N)$ and set
  \[
    x := \left( \frac{\delta}{m(y)}\right)^\frac{1}{p} y.
  \]
  Then, by using the homogeneity $\Pi_M^p (\lambda y) = \lambda \Pi_M^p(y)$ and $D_p(\lambda x, \lambda y) = \lambda^p D_p(x,y)$, we obtain
  \[
    \varepsilon > D_p(\Pi_{M\cap N}^p x,x) =\frac{\delta}{m(y)} D_p(\Pi_{M\cap N}^p y,y),
  \]
  which is equivalent to 
  \[
    D_p(\Pi_{M\cap N}^p y,y) < \frac{\varepsilon m(y)}{\delta}  = \frac{\varepsilon}{\delta} \max \{  D_p(\Pi_M^p y, y), D_p(\Pi_N^p y, y)\}.
  \]
  Since the above is true for arbitrary $\varepsilon>0$, we may set $\varepsilon =1$ and $\displaystyle \kappa := \frac{1}{\delta}$ which finishes the proof of linear Bregman regularity.\\
  Since (iii) is a weaker version of (ii) the implication (ii)$\Rightarrow$(iii) is immediate.\\
  We are left to prove (iii)$\Rightarrow$(ii). Denote by $B_p(0,1) = \{ x\in X \ | \ D_p(x,0)\leq 1\}$ and observe that $B_p(0,1)$ is a Bregman bounded subset of $X$. From assertion (iii) we derive the existence of a constant $\kappa>0$ such that
  \[
    D_p(A\cap B,x) \leq \kappa \max\{D_p(A,x),D_p(B,x)\} \quad \text{ for all } x\in B_p(0,1)
  \]
  Now for arbitrary nonzero $y\in X$ we set
  \[
    z:= \frac{y}{D_p(y,0)^\frac{1}{p}}
  \]
  and obtain
  \[
    D_p(M\cap N,z) \leq \kappa m(z) \Leftrightarrow \frac{1}{D_p(y,0)} D_p(M\cap N,y) \leq \frac{\kappa}{D_p(y,0)} m(y)
  \]
  which in turn implies $D_p(M\cap N,y) \leq \kappa m(y)$ as required.
\end{proof}

We will now show that there is a large class of pairs of spaces $(M, N)$ which are boundedly Bregman regular. In order to do so we use the following lemma.

\begin{lemma}\label{lem:BregToZeroDistToZero}
  Let $X$ be a uniformly convex and uniformly smooth Banach space and let $M\subset X$ be a closed linear subspace. Let $\{x_n\}_{n=0}^{\infty}$ be a sequence in $X$ with $D_p(M,x_n)\to 0$. Then also the metric distance $d(x_n,M)$ converges to zero.
\end{lemma}

\begin{proof}
  Setting $y := \Pi_M^p x  -x $ in~\eqref{eq:xu-roach-totalconv} allows us to conclude from Theorem \ref{thm:xu-roach}, that the inequality
  \begin{align}\label{thm:boundedlyreg-ineq1}
    D_p (\Pi_M^p x , x) \geq K \int_0^1 \frac{R^p}{t} \delta_X \left(\frac{t\|\Pi_M^p x - x \|}{2R}\right) \mathrm{d}t
  \end{align}
  holds for all $x\in X$.  Recall that $\delta_X(t) >0$ whenever $t > 0$ since $X$ is uniformly convex. This implies that the expression~$\delta_X (t)R^p / t$ is strictly positive for every $t>0$. Therefore
  \[
    K \int_0^1 \frac{R^p}{t} \delta_X \left(\frac{t\|\Pi_M^p x_n - x_n \|}{2R}\right) \mathrm{d}t \leq D_p(\Pi^p_Mx_n,x_n) \xrightarrow[n\to\infty]{} 0
  \]
  imlies that the argument of the integral has to converge to zero and therefore also
  \[
    \|\Pi_M^p x_n - x_n \| \xrightarrow[n\to\infty]{} 0.
  \]
  Finally using $\Pi_{M}x_n\in M$, we conclude that
  \[
    d(x_n, M) \leq     \|\Pi_M^p x_n - x_n \| \xrightarrow[n\to\infty]{} 0,
  \]
  as claimed.
\end{proof} 

\begin{proposition} \label{coro:closedSum-bdlyBregmanReg}
  Let $X$ be an uniformly convex and uniformly smooth Banach space and let $M,N\subset X$ be closed linear subspaces, such that $M+N$ is closed. Then $(M,N)$ is boundedly Bregman regular.
\end{proposition}

\begin{proof}
  Recall that by~\cite[p. 200, Corollary 4.5.]{Bauschke1993}, the pair $(M,N)$ is linearly regular, whenever the sum $M+N$ is closed. Therefore there is a $\kappa > 0$ such that
  \[
    \dist(M\cap N,x) \leq \kappa \max \{\dist(M,x),\dist(N,x)\}
  \]
  for all $x\in X$. Using Theorem \ref{thm:xu-roach}, we now argue that the inequality
  \begin{align}\label{thm:boundedlyreg-ineq2}
    D_p(M\cap N,x) \leq D_p(P_{M\cap N}x,x) \leq C \int_0^1 \frac{R^p}{t} \rho_X\left( \frac{t \|P_{M\cap N}x-x\|}{\|(1-t)x - tP_{M\cap N}x\|}\right) \mathrm{d}t
  \end{align}
  holds for all $x\in X$. This can be seen by setting $y := P_{M\cap N}x -x $ in \eqref{eq:xu-roach:smooth}.\\
  Now given a sequence $(x_n)_{n\in \mathbb{N}}$ satisfying
  \[
    \max \{ D_p(M ,x_n), D_p(N ,x_n) \} \underset{n\to\infty}{\longrightarrow} 0,
  \]
  we may use Lemma~\ref{lem:BregToZeroDistToZero} to obtain that $\dist(x_n,M)\to 0$. Similarly, $\dist(x_n,N)\to 0$ by replacing $M$ by $N$ in the above argument.
  Using the linear regularity of $(M,N)$  this means that $\|P_{M\cap N}x -x \| \xrightarrow[n\to\infty]{} 0$.
  Now combining this behaviour with~\eqref{thm:boundedlyreg-ineq2}, we may conclude that $D_p(M\cap N, x) \xrightarrow[n\to\infty]{} 0$ which finishes the proof.
\end{proof}

For Banach spaces $X$ which are in a certain sense very close to Hilbert spaces, we obtain an even stronger result.

\begin{proposition}\label{prop:lin-Bregman-reg-2sc}
  Let $X$ be $2$-convex and $2$-smooth and set $\rho = \sigma = p = 2$. Given two closed linear subspaces $M, N\subset X$ with closed sum. Then the pair $(M,N)$ is linear Bregman regular, i.e. there is a constant $\kappa>0$ such that
  \begin{align}\label{eq:lin-Bregman-reg}
    D_p(\Pi_{M\cap N}^p x, x) \leq \kappa \max \{D_p(\Pi_M^p x, x), D_p(\Pi_N^px, x)\}.
  \end{align}
  for all $x\in X$.
\end{proposition}

\begin{proof} From~\cite[p. 200, Corollary 4.5.]{Bauschke1993}, we know that $(M,N)$ is linear regular. 
  We use Corollary \ref{coro:bregman:sym:2-sc} in combination with Lemma~\ref{lem:bregmanproj-metricproj:ineq} to conclude that
  \begin{align*}
    D_2(\Pi_{M\cap N}^2 x, x) &\underset{\eqref{eq:bregmanproj-metricproj:1}}{\leq} C_1 \dist(M\cap N,x) ^2 \\
                              &\leq C_1 \kappa^2 \max\{ \dist(M,x) ^2, \dist( N,x) ^2\}\\
                              &\underset{\eqref{eq:bregmanproj-metricproj:2}}{\leq} C_1 \kappa^2 \max\{C_2^2 D_2(\Pi_M^2 x, x), C_2^2 D_2(\Pi_N^2 x, x) \} \\
                              &=  C_1 \kappa^2 C_2^2 \max\{D_p(\Pi_M^2 x, x), D_2(\Pi_N^2 x, x)\},
  \end{align*}
  for suitable constants $C_1,C_2>0$.
\end{proof}

We conclude this section with two examples which show that the above conditions are not vacuous.

\begin{example}
  We consider the subspaces
  \[
    M_1 = \langle \begin{pmatrix}1\\0\\0\end{pmatrix}, \begin{pmatrix}0\\1\\0\end{pmatrix}\rangle \qquad\text{and}\qquad M_2 = \langle \begin{pmatrix}1\\0\\0\end{pmatrix}, \begin{pmatrix}0\\0\\1\end{pmatrix}\rangle
  \]
  of $\ell^{3}_q$, $1<q<\infty$, $q\neq 2$. Note that the intersection $M_1\cap M_2$ is spanned by the first standard basis vector $e_1$. A direct computation shows that
  \[
    \Pi^q_{M_1}\begin{pmatrix}x\\y\\z\end{pmatrix} = \begin{pmatrix}x\\y\\0\end{pmatrix}, \qquad     \Pi^q_{M_2}\begin{pmatrix}x\\y\\z\end{pmatrix} = \begin{pmatrix}x\\0\\z\end{pmatrix} \qquad\text{and}\qquad \Pi^q_{M_1\cap M_2}\begin{pmatrix}x\\y\\z\end{pmatrix} = \begin{pmatrix}x\\0\\0\end{pmatrix}.
  \]
  Setting
  \[
    v = \begin{pmatrix}x\\y\\z\end{pmatrix},
  \]
  from the above we obtain
  \[
    D_q(M_1,v) = \left(1-\frac{1}{q}\right) |z|^q, \qquad D_q(M_2,v) = \left(1-\frac{1}{q}\right) |y|^q
  \]
  and
  \[
    D_q(M_1\cap M_2, v) = \left(1-\frac{1}{q}\right) \left(|y|^q+|z|^q\right).
  \]
  Therefore,
  \[
    D_q(M_1\cap M_2, v) \leq 2 \max\{D_q(M_1,v), D_q(M_2,v)\}
  \]
  which means that $(M_1,M_2)$ is linear Bregman regular. A similar computation (with an additional factor) shows that it is also linear Bregman regular for $p\neq q$.
\end{example}

\begin{example}
    We consider the subspaces
  \[
    M_1 = \langle \begin{pmatrix}1\\0\\\frac{1}{2}\end{pmatrix}, \begin{pmatrix}1\\1\\\frac{99}{100}\end{pmatrix}\rangle \qquad\text{and}\qquad M_2 = \langle \begin{pmatrix}1\\0\\\frac{1}{2}\end{pmatrix}, \begin{pmatrix}1\\1\\\frac{101}{100}\end{pmatrix}\rangle
  \]
  of $\ell^{3}_3$. In addition, we use the points
  \[
    v_\lambda := (1-\lambda)\begin{pmatrix}1\\0\\\frac{1}{2}\end{pmatrix}+\lambda\begin{pmatrix}1\\1\\1\end{pmatrix} \qquad\text{and}\qquad v:=\begin{pmatrix}1\\0\\\frac{1}{2}\end{pmatrix}.
  \]
  We obtain
  \begin{equation}\label{eq:toIntersection1}
    D_{3}(tv, v_\lambda) = \frac{2 \lambda^{3}}{3} + \frac{3 t^{3}}{8} - t - \frac{\left(\frac{\lambda}{2} + \frac{1}{2}\right)^{3}}{3} - \left(\frac{\lambda}{2} + \frac{1}{2}\right)^{2} \left(- \frac{\lambda}{2} + \frac{t}{2} - \frac{1}{2}\right) + \frac{2}{3}
  \end{equation}
  and hence
  \[
    \frac{\mathrm{d}}{\mathrm{d}t} D_{3}(tv, v_\lambda) =\frac{9 t^{2}}{8} - \frac{\left(\frac{\lambda}{2} + \frac{1}{2}\right)^{2}}{2} - 1
  \]
  which vanishes for
  \[
    t = \frac{\sqrt{\lambda^{2} + 2 \lambda + 9}}{3}
  \]
  since we may assume $t\geq 0$. If we insert this value into~\eqref{eq:toIntersection1}, we obtain
  \begin{equation}\label{eq:toIntersection2}
    \begin{aligned}
      D_{3}(\Pi^{3}_{M_1\cap M_2}v_\lambda, v_\lambda) &= \frac{3 \lambda^{3}}{4} - \frac{\lambda^{2} \sqrt{\lambda^{2} + 2 \lambda + 9}}{36} + \frac{\lambda^{2}}{4} - \frac{\lambda \sqrt{\lambda^{2} + 2 \lambda + 9}}{18} + \frac{\lambda}{4}\\ & \qquad\qquad  - \frac{\sqrt{\lambda^{2} + 2 \lambda + 9}}{4} + \frac{3}{4}
    \end{aligned}
  \end{equation}
  In addition, setting
  \[
    w_\lambda = \left(1-\frac{50}{49}\lambda\right)  \begin{pmatrix}1\\0\\\frac{1}{2}\end{pmatrix} + \frac{50}{49}\lambda \begin{pmatrix}1\\1\\\frac{99}{100}\end{pmatrix} \in M_1
  \]
  we obtain
  \[
    D_{3}(\Pi^{3}_{M_1}v_\lambda,v_\lambda) \leq D_{3}(w_\lambda, v_\lambda) = \frac{148 \lambda^{3}}{352947}
  \]
  and similarly by setting
  \[
    u_\lambda =  \left(1-\frac{50}{51}\lambda\right) \begin{pmatrix}1\\0\\\frac{1}{2}\end{pmatrix} + \frac{50}{51} \lambda \begin{pmatrix}1\\1\\\frac{101}{100}\end{pmatrix} \in M_2
  \]
  we get
  \[
    D_{3}(\Pi^{3}_{M_2}v_\lambda,v_\lambda) \leq D_{3}(u_\lambda, v_\lambda) = \frac{152 \lambda^{3}}{397953}.
  \]
  From these inequalities, we finally may conclude that
  \begin{multline*}
    \frac{D_p(\Pi_{M_1\cap M_2}^pv_\lambda,v_\lambda)}{\max\{D_p(\Pi_{M_1}^pv_\lambda,v_\lambda), D_p(\Pi_{M_2}^pv_\lambda,v_\lambda)\}}\\ \geq \frac{117649 \left(27 \lambda^{3} - \lambda^{2} \sqrt{\lambda^{2} + 2 \lambda + 9} + 9 \lambda^{2} - 2 \lambda \sqrt{\lambda^{2} + 2 \lambda + 9} + 9 \lambda - 9 \sqrt{\lambda^{2} + 2 \lambda + 9} + 27\right)}{1776 \lambda^{3}}
  \end{multline*}
  which goes to infinity for $\lambda\to0$ as can be seen by using L'H\^{o}pital's rule twice. This shows that the pair $(M_1, M_2)$ is not linear Bregman regular.
\end{example}
\section{Convergence behaviour of iterated Bregman projections}

As an important tool for the convergence analysis of sequenes we use the following concept of Bregman monotone sequences.
\begin{definition}
  A sequence $(x_n)_{n=0}^\infty$ in $X$ is called \emph{Bregman monotone} with respect to a subset $C\subset X$, if
  \[
    D_p(z,x_k) \leq D_p(z,x_l) 
  \]
  for all $k\geq l$ and all $z\in C$.
\end{definition}

Note that if a sequence $(x_n)_{n=0}^{\infty}$ is Bregman monotone with respect to a set containing zero, in hence in particular with respect to a linear subspace, it has to be bounded since
\begin{equation}\label{eq:BregmanMonBounded}
  \left(1-\frac{1}{p}\right) \|x_n\|^p = D_p(0, x_n) \leq D_p(0,x_0)
\end{equation}
for all $n\in\mathbb{N}$. Observe that this argument also shows that in this case the norm of the elements of the sequence satisfies $\|x_n\|\leq \|x_0\|$ for all $n\in\mathbb{N}$.

The next proposition shows that in the case of Banach spaces which are uniformly convex and uniformly smooth, Bregman monotonicity with respect to a closed linear subspace $M$ implies convergence if the distance to the subspace $M$ converges to zero.

\begin{proposition}\label{prop:BregMonConv}
  Let $X$ be a uniformly convex and uniformly smooth Banach space and let $(x_n)_{n=0}^\infty$ be a sequence which is Bregman-monotone with respect to a closed linear subspace $M$ of $X$. If $D_p(M,x_n) \xrightarrow[n\to\infty]{}0$, the $(x_n)_{n=0}^\infty$ converges to a point in $M$.
\end{proposition}

\begin{proof}
  Let us take a look at the case $n\geq m$. Recall that Bregman monotonicity implies that $(x_n)_{n=0}^\infty$ is bounded. Using the Bregman monotonicity and that the distance to $M$ is going to zero, we see that
  \begin{align} \label{eq:thm:bregmon:conv:1}
    D_p(\Pi_M^p x_m, x_n) \leq D_p(\Pi_M^p x_m, x_m) \xrightarrow[m\to\infty]{} 0 
  \end{align}
  which allows us to conclude that
  \[
    D_p(\Pi_M^p x_m, x_n) + D_p(\Pi_M^p x_m, x_m) \xrightarrow[m\to \infty]{} 0.
  \]
  We use the three-point identity 
  \begin{align}\label{eq:thm:bregmon:conv:2}
    D_p (x_n,x_m) = &D_p(x_n, \Pi_M^p x_m) + D_p(\Pi_M^p x_m, x_m) + \\
                    &\quad  +\langle j_p(\Pi_M^p x_m) - j_p(x_m), x_n - \Pi_M^p x_m \rangle \nonumber 
  \end{align}
  and observe that 
  \[
    |\langle j_p(\Pi_M^p x_m) - j_p(x_m), \Pi_M^p x_m - x_n \rangle| \leq \underbrace{\| x_n \|}_{\mathclap{\leq C \text{ constant}}} \cdot \underbrace{\|j_p(\Pi_M^p x_m) - j_p(x_m)\|}_{\xrightarrow[]{m\to\infty} 0} \to 0,
  \]
  since
  \[
    \langle j_p(\Pi_M^p x_m) - j_p(x_m), z\rangle = 0 
  \]
  for all $z\in M$. We are left to show that also the first summand converges to zero. Again using the Bregman monotonicity, observe that
  \[
    D_p(\Pi^p_{M}x_m,x_n) \leq D_p(\Pi^p_{M}x_m,x_m) \to 0
  \]
  for $m\to \infty$. Since the involded sequences are bounded, we may use the sequential consistency, i.e. assertion~(iv) of Proposition~\ref{prop:BregmanDistanceProperties}, together with assertion~(ii) of Proposition~\ref{prop:BregmanDistanceProperties} to conclude that
  also $D_p(x_n,\Pi^p_Mx_m)\to 0$ for $m\to\infty$. Hence, for all $\varepsilon>0$, we many pick an $N\in\mathbb{N}$ such that $D_p(x_n,x_m)<\varepsilon$ for all $n\geq m \geq N$.
  
  In order to be able to deduce that $(x_n)_{n=0}^{\infty}$ is a Cauchy sequence with respect to the Bregman distance, we need to show the corresponding inequality for the case $m>n$. For this case we again make use of assertions~(ii) and~(iv) of Proposition~\ref{prop:BregmanDistanceProperties} to exchange the arguments of the Bregman distance and to arrive at the first case again.

  This means that the sequence is a Cauchy sequence in the Bregman distance. Since by  Proposition~\ref{prop:BregmanDistanceProperties} the above implies that $(x_n)_{n=0}^{\infty}$ is a Cauchy sequence for the norm topology, it has a limit. The limit is contained in $\overline{M} = M$, since $M$ is a closed subspace and by Lemma~\ref{lem:BregToZeroDistToZero} the assumption $D_p(M,x_n)\to 0$ forces the (norm) distance $\operatorname{dist}(x_n,M)$ to converge to zero as well.
\end{proof}

\begin{remark}
  If the space $X$ is in addition $\rho$-convex and $\sigma$-smooth, the above proof together with Lemma~\ref{lem:bregmanproj-metricproj:ineq} gives us moreover some information on the convergence speed. In the situation above, denote by $x^*$ the limit of the sequence $(x_n)_{n=0}^{\infty}$, then using the three-point-identity we obtain
  \begin{equation}\label{eq:ConvRateVer0}
    \begin{aligned}
      D_p(x^*,x_n) =& D_p(x^*,\Pi_M^p x_n) + D_p(\Pi_M^p x_n, x_n) + \langle j_p(\Pi_M^p x_n) - j_p(x_n), x^* - \Pi_M^p x_n\rangle \\
      =& D_p(x^*,\Pi_M^p x_n) + D_p(\Pi_M^p x_n, x_n) \\
      \leq & C \, D_p(\Pi_M^p x_n, x^*)^\alpha + D_p(\Pi_M^p x_n, x_n) \\
      \leq & C\, D_p(\Pi_M^p x_n, x_n)^\alpha + D_p(\Pi_M^p x_n, x_n)
    \end{aligned}
  \end{equation}
  where the constants $C>0$ and $\alpha$ only depend on $X$ and on the initial point $x_0$.
  The second equality follows from the characterisation of the Bregman projection in terms of the annihilator of $M$, i.e. from the fact that
  \[
    \langle j_p(\Pi_M^p x_n) - j_p(x_n), z\rangle = 0
  \]
  for $z\in M$, see Proposition \ref{prop:bregmanprojAnnihilator}, and the fact that $x^*\in M$ and $\Pi_M^p x_n \in M$.
\end{remark}

Now we will apply these tools the problem of convergence of sequences generated by alternating Bregman projections. In order to do so, we first need to check that a sequence generated this way is Bregman monotone with respect to the intersection of the ranges.

\begin{lemma}\label{lem:altern-bregman:bregMon}
  Let $X$ be a uniformly convex and uniformly smooth Banach space and let $M,N\subset X$ be closed linear subspaces. For every $x\in X$ the sequence generated by
  \[
    x_0 = x, \qquad x_{2n+1} = \Pi_{M}^p x_{2n}, \qquad x_{2n} = \Pi_{N}^p x_{2n-1},
  \]
  is Bregman-monotone with respect to the closed linear subspace $M \cap N$.
\end{lemma}

\begin{proof}
  Assume without loss of generality that $x_n \in N$. Then by Proposition~\ref{prop:Bregman-SQNE} for every $z \in M \cap N$ the inequality
  \[
    D_p(z, \Pi_{M}^p x_n) \leq D_p(z, x_n) - D_p(x_n, \Pi_{M}^p x_n)
  \]
  holds. Moreover, since $D_p(x_n, \Pi_{M}^p x_n) > 0$, we get
  \[
    D_p(z, \Pi_{M}^p x_n) \leq D_p(z, x_n).
  \]
  The same holds true, if we switch the roles of $M$ and $N$. Therefore the generated sequence $(x_n)_{n=0}^{\infty}$ is a Bregman-monotone sequence with respect to~$M \cap N$.
\end{proof}

\begin{proposition}\label{prop:altern-bregman:Conv}
  Let $X$ be a uniformly convex and uniformly smooth Banach space and $M,N\subset X$ be closed linear subspaces whose sum is closed. For every $x\in X$ the sequence generated by
  \[
    x_0 = x, \qquad x_{2n+1} = \Pi_{M}^p x_{2n}, \qquad x_{2n} = \Pi_{N}^p x_{2n-1},
  \]
  converges to $\Pi_{M\cap N}^p x$.
\end{proposition}

\begin{proof}
  The main idea of the proof is to use Proposition~\ref{prop:BregMonConv} and to show that the distance to the intersection~$M\cap N$ is going to zero. Since Lemma~\ref{lem:altern-bregman:bregMon} shows that the generated sequence is Bregman-monotone we are left to show that the Bregman distance to the intersection is going to zero. We prove this by contradiction. Without loss of generality let $x_n \in M$. Assume the Bregman distance to the intersection is not going to zero, since $D_p(M\cap N, x_n)$ is a decreasing sequence, there is an $\varepsilon > 0$, such that
  \[
    D_p(M\cap N, x_n) \xrightarrow[n\to\infty]{} \varepsilon.
  \]
  Recall that by Proposition~\ref{prop:Bregman-SQNE} we obtain
  \[
    D_p(z, \Pi_N x_n) = D_p(z, x_n) - D_p(\Pi_N^p x_n, x_n)
  \]
  for all $z\in N$. Rearranging the terms and using that $x_{n+1} = \Pi_{N}x_n$, we obtain that
  \[
    D_p(\Pi_N x_n, x_n) = D_p(z, x_n)- D_p(z, x_{n+1})
  \]
  for all $z \in M\cap N$.
  We set $z = \Pi_{M\cap N}^p x_n$ and observe that since $\Pi_{M\cap N}x_n$ minimises the Bregman distance to $x_n$ we have
  \[
    D_p(\Pi_{M\cap N}^p x_n, x_n)- D_p(\Pi_{M\cap N}^p x_n, x_{n+1}) \leq D_p(\Pi_{M\cap N}^p x_n, x_n)- D_p(\Pi_{M\cap N}^p x_{n+1}, x_{n+1})
  \]
  which implies
  \[
    D_p(\Pi_N x_n, x_n) \leq D_p(\Pi_{M\cap N}^p x_n, x_n)- D_p(\Pi_{M\cap N}^p x_{n+1}, x_{n+1}).
  \]
  Since both terms on the right hand side converge to $\varepsilon$, we obtain that $D_p(N,x_n)$ has to converge to zero. By the same reasoning, with switched roles of $M$ and $N$, we get that also $D_p(M,x_n)\to 0$. Now recall that by Proposition~\ref{coro:closedSum-bdlyBregmanReg} the assumption that $M+N$ is closed implies that $(M,N)$ is boundedly Bregman regular. This means that
  \[
    \lim_{n\to\infty} \max\{D_p(M,x_n), D_p(N,x_n)\} = 0 \Rightarrow \lim_{n\to\infty} D_p(M\cap N, x_n) = 0,
  \]
  which contradicts $D_p(M\cap N, x_n) \xrightarrow[n\to\infty]{} \varepsilon > 0$. So we have that \[D_p(M\cap N, x_n) \xrightarrow[n\to\infty]{} 0.\]
  Hence we are able to apply Proposition~\ref{prop:BregMonConv} to conclude convergence of the sequence.

  In order to finish the proof, we are left to show that its limit is the Bregman projection of $x_0=x$ onto the intersection $M \cap N$. In order to do so, denote this limit by $x^*$ and observe that
  \begin{align*}
    |\langle j_p(x^*) - j_p(x), z-x^*\rangle| =& \left|\left\langle j_p(x^*) - j_p(x_N) - \sum_{k=1}^N j_p(x_k) - j_p(x_{k-1}),z-x^*\right\rangle\right| \\
    =& \langle j_p(x^*)-j_p(x_N), z-x^*\rangle \xrightarrow[]{N\to\infty} 0,
  \end{align*}
  where we used the characterisation of the Bregman projection onto $M$ by
  \[
    \langle j_p(\Pi_M^p x) - j_p(x), z \rangle = 0,
  \]
  for all $z\in M$, see Proposition~\ref{prop:bregmanprojAnnihilator}. Hence
  \[
    \langle j_p(x^*) - j_p(x), z-x^*\rangle = 0
  \]
  for all $z\in M\cap N$, which implies that $x^*=\Pi^p_{M\cap N}x$ again by Proposition~\ref{prop:bregmanprojAnnihilator}
\end{proof}

The assumption that $M+N$ should be a closed subspace was used to obtain that the pair $(M,N)$ was boundedly Bregman regular. If we require $(M,N)$ to satisfy stronger regularity conditions, we are able to obtain some information on the rate of convergence of the method of alternating Bregman projections.

\begin{theorem}\label{thm:bregmanproj:convrate}
  Let $X$ be uniformly convex and uniformly smooth Banach space. Furthermore let $M,N$ be closed linear subspaces such that $(M,N)$ is (boundedly) linear Bregman regular. Then for every $x\in X$ the sequence defined by
  \[
    x_0 := x, \qquad x_{2n+1} := \Pi^p_{M}x_{2n}\quad\text{and}\quad x_{2n} := \Pi^p_{N}x_{2n-1}
  \]
  converges to $\Pi_{M\cap N}^px$. If in $X$ is an addition $\rho$-convex and $\sigma$-smooth and~$p$ such that $\sigma \leq p \leq \rho$, there are constants $C>0$ and $q\in (0,1)$ such that the sequence satisfies
  \begin{equation}
    D_p(\Pi^p_{M\cap N}x,x_n) \leq C q^n 
  \end{equation}
  and therefore also
  \begin{equation}
    \left\|\Pi^p_{M\cap N}x-x_n\right\| \leq C' (q')^{n}
  \end{equation}
  for some other constants $C'>0$ and $q'\in (0,1)$, that is, it converges linearly.
\end{theorem}

\begin{proof}
  For a given $n\in \mathbb{N}$ we may assume without loss of generality that $x_n\in N$ and therefore $x_{n+1} = \Pi_M^p x_n$. Since by Proposition~\ref{prop:Bregman-SQNE} the Bregman projections onto $M$ and onto $N$ are Bregman strongly quasi-nonexpansive,  we have
  \begin{align*}
    D_p(z, \Pi_M^p x) \leq D_p(z,x) - D_p(\Pi_M^p x, x) \qquad \forall z \in M \\
    D_p(z, \Pi_N^p x) \leq D_p(z,x) - D_p(\Pi_N^p x, x) \qquad \forall z \in N
  \end{align*}
  Setting $z=\Pi^{p}_{M\cap N}x_{n}$ and $x=x_n$ and since $\Pi_{M\cap N}^p x_{n+1}$ minimizes the Bregman distance to $x_{n+1}$ we obtain
  \begin{align} \label{eq:bregman:convrate:1}
    D_p(\Pi_{M\cap N}^p x_n,x_{n+1}) &\leq D_p(\Pi_{M\cap N}^p x_n,x_n) - D_p(\Pi_M^p x_n, x_n)\\
    D_p(\Pi_{M\cap N}^p x_n,x_{n+1}) &\geq D_p(\Pi_{M\cap N}^p x_{n+1}, x_{n+1}). \nonumber
  \end{align}
  We denote by $\kappa$ the constant in the inequality for the bounded linear Bregman regularity of $(M, N)$ over a suitably large ball containing all elements of the sequence~$(x_n)_{n=0}^{\infty}$. Note that by~\eqref{eq:BregmanMonBounded} it is sufficient to take a ball of radius $\|x_0\|$ around zero. Since $D_p(\Pi_M^p x_n,x_n) > D_p(\Pi_N^p x_n, x_n) = 0$, using the linear Bregman regularity, we may deduce that
  \[
    D_p(\Pi_{M\cap N}^p x_n, x_n) \leq \kappa \cdot D_p(\Pi_M^p x_n, x_n),
  \]
  which gives us in return
  \[
    D_p(\Pi_M^p x_n, x_n)\geq \frac{1}{\kappa} D_p(\Pi_{M\cap N}^p x_n, x_n).
  \]
  Observe that $\kappa \geq 1$, since $M\cap N \subset M$.
  Using this in \eqref{eq:bregman:convrate:1} we obtain
  \[
    D_p(\Pi_{M\cap N}^p x_{n+1}, x_{n+1}) \leq \left(1-\frac{1}{\kappa}\right) D_p(\Pi_{M\cap N}^p x_{n}, x_{n}).
  \]
  Now using this argument inductively we may conclude that
  \[
    D_p(\Pi_{M\cap N}^p x_{n+1}, x_{n+1}) \leq \left( 1-C\right)^{n+1} D_p(\Pi_{M\cap N}^p x_0, x_0).
  \]
  Using that $x_n$ is a Bregman regular sequence since the Bregman projection is strongly Bregman quasi-nonexpansive, we use Proposition~\ref{prop:BregMonConv} to conclude that $x_n$ converges to some $x^*\in M\cap N$. Since the assumptions of Proposition~\ref{prop:altern-bregman:Conv} are satisfied, we may conclude that the limit point~$x^*$ is indeed the Bregman projection of $x$ onto $M\cap N$.

  We are left to show that quantitative part for the case where $X$ is convex and smooth of power type. Picking $n$ sufficiently large so that the above expression is at most one, we may plug it into~\eqref{eq:ConvRateVer0}\  to obtain 
  \[
    D_p(\Pi_{M\cap N}^px, x_n) \leq 2 C' D_p(\Pi_{M\cap N}^px_n,x_n)^\alpha \leq C q^n 
  \]
  where
  \[
    q:= (1-C)^\alpha\in (0,1)\qquad\text{and}\qquad C = 2C' D_p(\Pi_{M\cap N}^p x_0, x_0)^\alpha
  \]
  which proves the first claim. Combining the above inequality with~\eqref{eq:bregman:sym:2} yields the second claim.  
\end{proof}

Since in $2$-convex and $2$-smooth Banach spaces, the condition that $M+N$ is closed already implies that $(M,N)$ is linear Bregman regular, we obtain the following corollary.

\begin{corollary}\label{thm:bregmanproj:convrate-2-sc}
  Let $X$ be a $2$-smooth and $2$-convex Banach space and let $M,N\subset X$ be closed linear subspaces, such that $M+N$ is closed. Then for every $x\in X$ the sequence defined by
  \[
    x_0 := x, \qquad x_{2n+1} := \Pi^2_{M}x_{2n}\quad\text{and}\quad x_{2n} := \Pi^2_{N}x_{2n-1}
  \]
  converges to $\Pi_{M\cap N}^2x$. Moreover there are constants $C>0$ and $q\in (0,1)$ such that the sequence satisfies
  \begin{equation}
    D_2(\Pi^2_{M\cap N}x,x_n) \leq C q^n 
  \end{equation}
  and therefore also
  \begin{equation}
    \left\|\Pi^2_{M\cap N}x-x_n\right\| \leq C' (q')^{n}
  \end{equation}
  for some other constants $C'>0$ and $q'\in (0,1)$, i.e. the sequence converges linearly.
\end{corollary}

\begin{proof}
  Since we have $\sigma = p = \rho = 2$ we use Proposition~\ref{prop:lin-Bregman-reg-2sc} to obtain that $(M,N)$ is linear Bregman regular with constant $\kappa >0$. Now Theorem~\ref{thm:bregmanproj:convrate} implies the claimed statement.
\end{proof}

\section{Convergence properties of the alternating algorithm}

The main tool which we will be using in this section is the following decoposition theorem which is due to Y.~Alber.

\begin{theorem}\label{thm:alber:decomp}
  Let $X$ be a uniformly convex and uniformly smooth Banach space and let $M\subset X$ be a closed linear subspace. Morover let $1<p<\infty$ and $\frac{1}{p}+\frac{1}{p^*}=1$. Then,
  \[
    x = P_M x + j_{p^*}\left( \Pi_{M^\perp}^{p^*} j_p(x)\right).
  \]
  for every $x\in X$.
\end{theorem}
\begin{proof}
  See Theorem~1.3 in~\cite[p.~332]{Alber2005}.
\end{proof}

Rearranging the terms in the decomposition given by the theorem above, we arrive at
\begin{equation}\label{eq:BregmanMetric}
  (I-P_M)x = j_{p^*}(\Pi^{p^*}_{M^\perp}j_p(x)).
\end{equation}
Using the fact that the duality mappings $j_p$ and $j_{p^*}$ are inverse to each other, we may conclude that
\[
  (I-P_{N})(I-P_M)x = j_{p^*}(\Pi^{p^*}_{N^\perp}\Pi^{p^*}_{M^\perp}j_p(x)).
\]
for every pair $(M,N)$ of closed linear subspaces.

In order to use the results of the previous section for the problem of convergence of the alternating algorithm, we need the following well-known connection between the closedness of the subspace $M+N$ and of $M^\perp+N^\perp$.

\begin{proposition}\label{prop:ClosedSumEquivalence}
  Let $X$ be a Banach space and $M,N\subset X$ be closed linear subspaces. The sum $M+N$ is closed if and only if $M^\perp+N^\perp$ is closed.
\end{proposition}

\begin{proof}
  See, e.g. Theorem~4.8 in~\cite[p.~221]{Kato1976}.
\end{proof}

Using this connection, we can now recover the known theorem on the convergence of the alternating algorithm from our results on the iteration of Bregman projections.

\begin{corollary}[{Theorem~10 in~\cite[p.,~93]{Deutsch1979}, cf. Theorem~2.2 in~\cite[p.~749]{Pinkus2015}}]
  \label{cor:ConvergenceAlternatingAlgorithm}
  Let $X$ be a uniformly convex and uniformly smooth Banach space. Moreover let $M, N\subset X$ be closed linear subspaces with closed sum $M+N$. Then for every $x\in X$ the sequence defined by
  \[
    x_0 := x, \qquad x_{2n+1} = (I-P_M)x_{2n}, \qquad x_{2n} = (I-P_N)x_{2n-1}
  \]
  converges to $(I-P_{M+N})x$.
\end{corollary}

\begin{proof}
  Note that by Proposition~\ref{prop:ClosedSumEquivalence} the sum $M^\perp+N^\perp$ is closed. Therefore we may apply Proposition~\ref{prop:altern-bregman:Conv} to obtain that the sequence defined by
  \[
    y_0 = j_p(x), \qquad y_{2n+1} = \Pi_{M^\perp}^ {p^*} y_{2n}, \qquad y_{2n} = \Pi_{N^\perp}^{p^*} y_{2n-1},
  \]
  converges to $\Pi_{M^\perp\cap N^\perp}^{p^*} j_p(x)$. Moreover note that by~\eqref{eq:BregmanMetric} we have
  $x_n = j_{p^*}(y_n)$ and hence the continuity of the duality mapping implies that
  \[
    \lim_{n\to\infty} x_n = j_{p^*}(\Pi^{p*}_{M\perp\cap N^\perp} j_p(x)) = (I-P_{(M^\perp\cap N^\perp)^\perp})x
  \]
  where the last equality follows again from~\eqref{eq:BregmanMetric}. Noting that $(M^\perp\cap N^\perp)^\perp = M+N$ finishes the proof.
\end{proof}

Using the result on the rate of convergence of the alternating Bregman projections, we are also able to give a result on the convergence speed of the alternating algorithm.

\begin{proposition}
  Let $X$ be $\rho$-convex and $\sigma$-smooth and let $\sigma \leq p, p^* \leq \rho$ with $\frac{1}{p}+ \frac{1}{p^*} = 1$. Further let $M,N\subset X$ be closed linear subspaces such that their sum is closed and for every bounded subset $S\subset X$, there is $\kappa\geq 1$ such that
  \begin{equation}\label{eq:DualLinBregReg}
    \min\{d(x,M)^{p}, d(x,N)^{p}\} \leq \frac{\kappa-1}{\kappa} \|x\|^p + \frac{1}{\kappa} d(x,M+N)^p
  \end{equation}
  is satisfied for all $x\in S$. Then, for every $x\in X$ the sequence defined by
  \[
    x_0 := x, \qquad x_{2n+1} := (I-P_M)x_{2n}\quad\text{and}\quad x_{2n} := (I-P_N)x_{2n-1}
  \]
  converges to $(I-P_{M + N})x$. Moreover there are constants $C>0$ and $q\in (0,1)$ such that the sequence satisfies
  \begin{equation}
    \left\|(I-P_{M + N})x-x_n\right\| \leq C q^{n},
  \end{equation}
  that is, it converges linearly.
\end{proposition}

\begin{proof}
  First, as in the proof of Corollary~\ref{cor:ConvergenceAlternatingAlgorithm}, observe that by setting
  \[
    y_0 = j_p(x), \qquad y_{2n+1} = \Pi_{M^\perp}^ {p^*} y_{2n}, \qquad y_{2n} = \Pi_{N^\perp}^{p^*} y_{2n-1},
  \]
  we obtain~$x_n = j_{p^*}(y_n)$ from~\eqref{eq:BregmanMetric}. We now want to apply Theorem~\ref{thm:bregmanproj:convrate} to obtain the desired convergence rate. In order to be able to do so, we show that~\eqref{eq:DualLinBregReg} translates to the pair $(M^\perp,N^\perp)$ being boundedly linear Bregman regular. Observe first that for $y=j_p(x)$ we obtain
  \begin{equation}\label{eq:BregDistDualCond}
    \begin{aligned}
      D_{p^{*}}(M^\perp,y) &= D_{p^{*}}(\Pi^{p^*}_{M^\perp}y,y) = D_{p}(j_{p^{*}}(\Pi^{p^*}_{M^\perp}y), j_{p^{*}}(y)) = D_p(x, x-P_Mx) \\
      &= \frac{1}{p} \|x\|^p - \frac{1}{p} \|x-P_Mx\|^p - \langle j_p(x-P_Mx), P_Mx\rangle \\
      &= \frac{1}{p} \|x\|^p - \frac{1}{p} \|x-P_Mx\|^p = \frac{1}{p} \|x\|^p - \frac{1}{p} d(x,M)^p
    \end{aligned}
  \end{equation}
  from~\eqref{eq:BregmanMetric}, Proposition~\ref{prop:BregmanDistanceProperties} and the simple observation that
  \[
    \langle j_p(x-P_Mx), P_Mx\rangle = \langle \Pi^{p^*}_{M^\perp}j_p(x),P_Mx\rangle = 0
  \]
  where we again used~\eqref{eq:BregmanMetric}. Rearranging the terms in~\eqref{eq:DualLinBregReg} leads to
  \[
    \|x\|^p - d(x,M+N)^p \leq \kappa \left(\|x\|^p-\min\{d(x,M)^p,d(x,N)^p\}\right)
  \]
  which we may further rearrange to
  \[
    \frac{1}{p}\|x\|^p - \frac{1}{p} d(x,M+N)^p \leq \kappa \max\left\{\frac{1}{p}\|x\|^p-\frac{1}{p}d(x,M)^p,\frac{1}{p}\|x\|^p-\frac{1}{p}d(x,N)^p\right\}
  \]
  Inserting~\eqref{eq:BregDistDualCond} and the corresponding relations for $N$ and $M+N$ we obtain
  \[
    D_{p^*}(M^\perp\cap N^\perp, j_p(x)) \leq \kappa \max\left\{D_{p^*}(M^\perp, j_p(x)),D_{p^*}(N^\perp, j_p(x))\right\}
  \]
  which is the inequality required for linear Bregman regularity. It remains to be shown that every bounded subset of $X^*$ can be realised as the image of a bounded subset of $X$ via the duality mapping~$j_p$. With this aim, first recall that the duality mapping of a reflexive space is bijective and that uniformly convex and uniformly smooth spaces are reflexive. Since the duality mapping on uniformly smooth spaces is uniformly continuous on bounded sets, it also maps bounded sets into bounded sets which finishes the proof of the claim that $(M^\perp, N^\perp)$ is boundedly linear Bregman regular.

  Now, using Theorem~\ref{thm:bregmanproj:convrate} and Proposition~\ref{coro:bregman:sym}, we obtain
  the existence of constants $C_1, C_2>0$ and $q'\in (0,1)$ such that
  \begin{align*}
    \|(I-P_{M+N})x-x_n\| &\leq C_1 D_p(x_n,(I-(P_{M+N}))x)^{1/\rho} = C_1 D_{p^*}(\Pi^{p^*}_{M\perp\cap N^\perp}y_0,y_n)^{1/\rho}\\
    &\leq C_1C_2 (q')^{n/\rho} =: C q^{n} 
  \end{align*}
  as claimed since $q=(q')^{1/\rho}\in (0,1)$.
\end{proof}

Since the dual space of a $2$-convex and $2$-smooth space is also $2$-convex and $2$-smooth, we obtain the following corollary.

\begin{corollary}
  Let $X$ be $2$-convex and $2$-smooth and let $M,N\subset X$ be closed linear subspaces such that their sum is closed. Then, for every $x\in X$ the sequence defined by
  \[
    x_0 := x, \qquad x_{2n+1} := (I-P_M)x_{2n}\quad\text{and}\quad x_{2n} := (I-P_N)x_{2n-1}
  \]
  converges to $(I-P_{M + N})x$. Moreover there are constants $C>0$ and $q\in (0,1)$ such that the sequence satisfies
  \begin{equation}
    \left\|(I-P_{M + N})x-x_n\right\| \leq C q^{n},
  \end{equation}
  i.e. the convergence is linear.
\end{corollary}

\bigskip
{\noindent\textbf{\textsf{Acknowledgement.}}} This research was funded by the Tyrolean Science Fund (Tiroler Wissenschaftsförderung). The authors wish to thank Simeon Reich and Rafał Zalas for  interesting discussions on this topic. In addition, we thank Simeon Reich for carefully reading the article and for his helpful comments.

\vspace{8mm}
\noindent
Christian Bargetz\\
Department of Mathematics\\
University of Innsbruck\\
Technikerstraße 13,
6020 Innsbruck,
Austria\\
\texttt{christian.bargetz@uibk.ac.at}\\[4mm]
\noindent
Emir Medjic\\
Department of Mathematics\\
University of Innsbruck\\
Technikerstraße 13,
6020 Innsbruck,
Austria\\
\texttt{emir.medjic@uibk.ac.at}
\end{document}